\documentclass[a4paper,11pt]{amsart}

\usepackage{graphicx}
\usepackage{pst-grad}
\usepackage{url}
\usepackage{pstricks}
\usepackage[utf8]{inputenc}
\usepackage{pstricks-add}
\usepackage{pgf}
\usepackage{tikz-cd}
\usepackage[thinlines]{easytable}
\usepackage[all]{xy}
\usepackage{pinlabel}
\usepackage{psfrag}
\usepackage{hyperref}
\usepackage{float}
\usepackage{color}
\usepackage{tikz}

\usepackage{amsmath, amsfonts, amsthm, amssymb, amscd, mathtools, tensor}
 \newtheorem{thm}{Theorem}[section]
 \newtheorem{prop}[thm]{Proposition}

\theoremstyle{definition}
 \newtheorem{rem}[thm]{Remark}
 \newtheorem{definition}{Definition}[section]
\numberwithin{equation}{section}

\newcommand\mj{\mbox{\bf 1}}

\def\d#1{{#1\kern-0.4em\char"16\kern-0.1em}}
\def\D#1{{\raise0.2ex\hbox{-}\kern-0.4em#1}}
\def \Dj{\mbox{\raise0.3ex\hbox{-}\kern-0.4em D}}
\definecolor{britishracingzelena1}{rgb}{0.0, 0.26, 0.15}
\def\zn{,\kern-0.09em,}

\newrgbcolor{zzttqq}{0.8 0.8 0.8}
\newrgbcolor{svetla}{.9 .9 .9}
\newrgbcolor{siva1}{.1 .1 .1}
\newrgbcolor{siva2}{.2 .3 .2}
\newrgbcolor{siva3}{.7 .7 .7}
\newrgbcolor{plava}{.67 .84 .9}
\newrgbcolor{zelena}{.56 .93 .56}
\newrgbcolor{zelena1}{.32 .73 .32}

\title{A diagrammatic presentation of the category 3Cob}

\author[Nikoli\' c]{Jovana Nikoli\' c$^{\ast}$}
\address{\scriptsize{Faculty of Mathematics, University of Belgrade}}
\email{jovana.nikolic@matf.bg.ac.rs}
\thanks{$^\ast$Corresponding author email: jovana.nikolic@matf.bg.ac.rs}
\author[Petri\' c]{Zoran Petri\' c}
\address{\scriptsize{Mathematical Institute SANU\\ Knez Mihailova 36, p.f.\ 367\\ 11001 Belgrade, Serbia}}
\email{zpetric@mi.sanu.ac.rs}
\author[Zeki\' c]{Mladen Zeki\' c}
\address{\scriptsize{Mathematical Institute SANU\\ Knez Mihailova 36, p.f.\ 367\\ 11001 Belgrade, Serbia}}
\email{mzekic@mi.sanu.ac.rs}

\date{}

\begin{document}

\begin{abstract}
A category equivalent to the category of 3-dimensional cobordisms is defined in terms of planar diagrams. The operation of composition in this category is completely described via these diagrams.

\vspace{.3cm}


\vspace{.5ex}

\noindent {\small {\it Keywords$\,$}: knots, links, surgery, Kirby's calculus, 3-manifolds with boundary, gluing}
\end{abstract}

\maketitle
\section{Introduction}

The category 3Cob has 2-dimensional, closed, oriented manifolds as objects and 3-dimensional cobordisms as arrows.
By a diagrammatic presentation of this category, we mean the following three things:
\begin{enumerate}
\item a language of diagrams with expressive power sufficient to present all the arrows of 3Cob;
\item a complete calculus telling us whether two diagrams present the same cobordism;
\item an operation on diagrams that corresponds to the composition of cobordisms.
\end{enumerate}

This paper covers all the above. Our language is based on the surgery description of closed manifolds introduced by Wallace, \cite{W60} and Lickorish, \cite{L62}. There are other diagrammatical languages for 3Cob. For the standard category of cobordisms, the most important is the one introduced by Turaev, \cite{T10}, and we mention also the work of Kerler, \cite{K99}, Sawin, \cite{S04}, and Juhasz, \cite{J18}. Our intention was to make an extension of the language of surgery in a form as simple as possible. This language is introduced in Section~\ref{standard diagrams} and its interpretation is explained in Section~\ref{interpretation}. Our diagrams consist just of circles and wedges of circles. The interpretation of such a diagram as a cobordism is straightforward.

We use the results from \cite{FGOP} to establish a calculus of moves, analogous to Kirby's calculus, \cite{K78}, which is complete in the sense that two diagrams present the same cobordism if and only if there is a finite sequence of moves transforming one diagram into the other. A discussion on such calculi is given in Section~\ref{calculus}.

The main topic of this paper is how to ``compose'' the diagrams. We are aware of Sawin's paper \cite{S04}, where a composition of diagrams is presented, in a very elegant way, by a sketch in Figure~9. We tried, but could not prove its correctness in the context of our diagrams, and when we tested the mending rule it went wrong. Maybe it is just a matter of misunderstanding. However, we believe that composition of diagrams requires more subtleties, and we present our approach in Section~\ref{composition}, while Section~\ref{sigmags1} serves to prepare the ground for this.

Such a diagrammatic presentation of 3Cob is important for us since our ongoing project is to establish how faithful a 3-dimensional TQFT could be. We believe that the simplicity of this presentation could make construction and analysis of 3-dimensional TQFT's more available. At least, it could shed a new light to these matters.

As a side product of our investigations, one finds possibility to use our diagrams for coherence questions in category theory. Sometimes diagrammatical (or graphical) languages used to express some coherence results combine graphs with boxes containing some extra information. For example, such are graphs related to categorical quantum protocols (see \cite{S07} and \cite{DPZ}). By replacing 1-dimensional strings (1-dimensional cobordisms) with 3-dimensional cobordisms whose boundary consists of two components of the same genus greater than 0, one can skip the role of boxes and present everything in completely geometrical terms.

We give an example to illustrate how we present cobordisms by diagrams and how we compose them. Consider the two diagrams illustrated in Figure~\ref{CD}.

\begin{figure}[!h]
\centering
\includegraphics{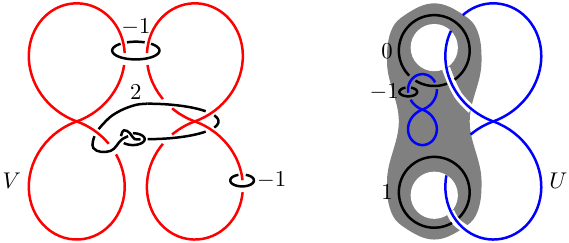}
\caption{The cobordisms $D$ and $C$}
\label{CD}
\end{figure}

The left-hand side diagram  is interpreted as a manifold so that two wedges of circles are thickened and their interiors are removed from $S^3$ in order to form two boundary components (genus 2 surfaces). Then a surgery according to the framed link consisting of three components is performed in a standard manner. The red colour of wedges indicates that the boundary components should be incoming. For the embedding of the source of the cobordism $D$ presented by this diagram it matters how the two boundary components could be identified, and this identification is ``on the nose'', i.e., what comes to mind first, according to the shapes of the corresponding wedges. (All this will become precise in Sections~\ref{standard diagrams} and \ref{interpretation}.)

The interpretation of the right-hand side diagram is analogous, apart from the fact that the blue colour of wedges of circles indicates that the corresponding boundary components should be outgoing. Again, the embedding of the target of the cobordism $C$ presented by this diagram will be precisely defined later, and at this point it only matters that the following diagram presents the identity cobordism on a genus 2 surface.
\begin{figure}[!h]
\centering
\includegraphics{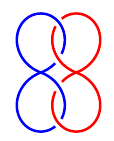}
\end{figure}

The composition $D\circ C$ (in which the boundary components corresponding to wedges labeled by $U$ and $V$ are identified, and analogously for unlabeled wedges) results in a closed manifold and its diagram is obtained in the following manner. We start with gluing two cobordisms along the boundary components labeled by $U$ and $V$. This is done by placing the diagram for $C$ in a handlebody indicated by shaded zone in Figure~\ref{CD}, and this handlebody is linked with the blue wedge labeled by $U$ as illustrated in this figure. This handlebody together with the diagram inside is moved so to form a neighbourhood of the wedge marked by $V$ in the diagram for $D$. In this way we obtain the diagram illustrated at the left-hand side of Figure~\ref{E}.

\begin{figure}[!h]
\centering
\includegraphics{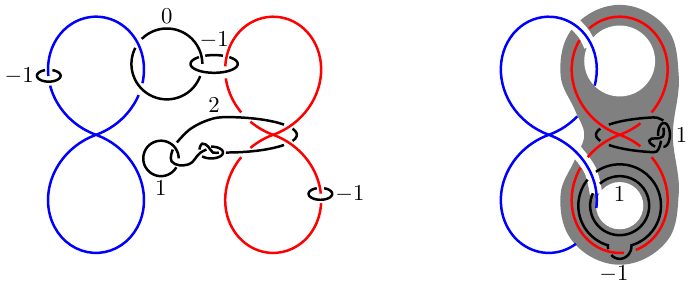}
\caption{Gluing $C$ and $D$ along $U$ and $V$}
\label{E}
\end{figure}

By using a variant of Kirby's calculus (see Section~\ref{calculus}), this diagram is transformed so that the blue and the red wedge are linked as in the diagram for the identity shown above (see the right-hand side of Figure~\ref{E}). The cobordism presented by this diagram has one incoming and one outgoing boundary component (marked in red and blue, respectively) and $D\circ C$ is a result of self-gluing along these boundaries. The diagram presenting $D\circ C$ is obtained by removing these two wedges and by inserting a diagram of the form illustrated in Figure~\ref{borromean}.
\begin{figure}[h!]
\centering
\includegraphics{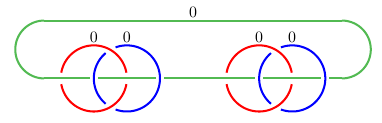}
\caption{} \label{borromean}
\end{figure}

The insertion of this diagram has to respect the shape of the handlebody presented by the shaded region in the diagram at the right-hand side of Figure~\ref{E}. Our diagram for the closed manifold $D\circ C$ is illustrated in Figure~\ref{F}.

\begin{figure}[!h]
\centering
\includegraphics[width=3.2cm]{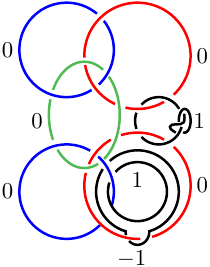}
\caption{A diagram for $D\circ C$}
\label{F}
\end{figure}

By a \emph{manifold} we mean a compact and oriented 3-manifold possibly with boundary. If not specified, one may assume that it is connected. We consider $S^3$ as an Alexandroff (one-point) compactification of $R^3$. The orientation of $S^3$ is fixed and it is assumed that the induced orientation of $R^3$ is right-handed. We presuppose some basic knowledge in surgery of manifolds which can be found in \cite{L62,R76,PS}.

\section{The category 3Cob}\label{diagrammatics}
The objects of the category 3Cob are closed, oriented surfaces. The arrows are 3-dimensional cobordisms consisting of a manifold $M$  together with two closed, oriented surfaces $\Xi_0$ (the \emph{source}) and $\Xi_1$ (the \emph{target}) and two embeddings $\varphi_0\colon \Xi_0\to M$ and $\varphi_1\colon \Xi_1\to M$,  whose images are disjoint and constitute $\partial M$. Moreover, taking the orientation of $\partial M$ to be induced by the orientation of $M$, $\varphi_0$ is orientation preserving and $\varphi_1$ is orientation reversing. We call the image of $\varphi_0$ the \emph{incoming} boundary, and the image of $\varphi_1$ the \emph{outgoing} boundary of $M$.

Two cobordisms $(M,\varphi_0,\varphi_1)$ and $(M',\varphi'_0,\varphi'_1)$ with the same source and target are considered to be equal when there exists an orientation preserving homeomorphism $w\colon M\to M'$ such that the following diagram commutes.
\begin{center}
\begin{tikzcd}
\Xi_0\arrow[hook]{r}{\varphi_0}\arrow[hook]{dr}[swap]{\varphi'_0}
&M \arrow{d}{w} \arrow[hookleftarrow]{r}{\varphi_1} & \Xi_1 \arrow[hook]{ld}{\varphi'_1}\\
&M'
\end{tikzcd}
\end{center}

\begin{rem}\label{isotopy1}
If $(M,\varphi_0,\varphi_1)$ and $(M,\varphi'_0,\varphi'_1)$ are such that $\varphi'_0\circ \varphi^{-1}_0$ and $\varphi'_1\circ \varphi^{-1}_1$ are isotopic to the identities, then these two cobordisms are equal.
\end{rem}

If the target of $(M,\varphi_0,\varphi)$ and the source of $(N,\theta,\theta_1)$ is the same surface $\Xi$, then their \emph{composition} is the cobordism whose underlying manifold is $(M\sqcup N)/_\sim$, where $\sim$ is such that
\begin{equation}\label{sim}
   \forall x\in\Xi \quad (\varphi(x),1)\sim (\theta(x),2).
\end{equation}
The source and the target embeddings of the resulting cobordism are derived from $\varphi_0$, $\theta_1$ and embeddings of $M$ and $N$ into $(M\sqcup N)/_\sim$.

For every object $\Xi$, the \emph{identity} arrow $\mj_\Xi\colon \Xi\to\Xi$ is defined as
\begin{center}
\begin{tikzcd}
\Xi\arrow[hook]{r}{\varphi_0}
&\Xi\times I \arrow[hookleftarrow]{r}{\varphi_1} & \Xi,
\end{tikzcd}
\end{center}
where for every $x\in\Xi$, $\varphi_0(x)=(x,0)$ and $\varphi_1(x)=(x,1)$. According to the product orientation, the embedding $\varphi_1$ is orientation reversing. The category 3Cob is equipped with a symmetric monoidal structure in which the tensor product is the disjoint union.

Let $C$ and $D$ be cobordisms whose underlying manifolds are $M_C$ and $M_D$, respectively. Let $\Sigma$ be a closed surface common to the target of $C$ and to the source of $D$, such that $\varphi\colon\Sigma\to M_C$ and $\theta\colon\Sigma\to M_D$ are parts of the target and the source embeddings. We define \emph{gluing} of $C$ and $D$ along the outgoing and incoming components of boundaries corresponding to $\varphi$ and $\theta$ to be the cobordism whose underlying manifold is $(M_C\sqcup M_D)/_\sim$, where $\sim$ is such that
\begin{equation}\label{sim}
   \forall x\in\Sigma \quad (\varphi(x),1)\sim (\theta(x),2).
\end{equation}
The source and the target embeddings of the resulting cobordism are derived from the source and the target embeddings of $C$ and $D$ (with $\theta,\varphi$ omitted) and embeddings of $M$ and $N$ into $(M\sqcup N)/_\sim$. It is obvious that composition is just a special case of gluing. On the other hand, every gluing could be performed by using tensor product with identities, symmetry and composition.

Let $C$ be a cobordism, with $M$ as the underlying manifold, such that a closed surface $\Sigma$ occurs as a part of its source and its target. Let $\varphi\colon\Sigma\to M$ and $\theta\colon\Sigma\to M$ be parts of the target and the source embeddings. We define \emph{mending} of $C$, along the outgoing and incoming components of the boundary corresponding to $\varphi$ and $\theta$, to be the cobordism whose underlying manifold is $M/_\sim$, where $\sim$ is such that
\begin{equation}\label{sim}
   \forall x\in\Sigma \quad \varphi(x)\sim \theta(x).
\end{equation}
The source and the target embeddings of the resulting cobordism are obtained by removing $\theta$ and $\varphi$ from the corresponding embeddings of the source and the target of $C$.

\vspace{2ex}

For every $g\geq 0$, we specify one closed, connected and oriented surface $\Sigma_g$ of genus $g$. Every object of 3Cob is isomorphic to a  finite sequence of such chosen surfaces (here we rely on the amphicheiral nature of surfaces). Hence by restricting the objects of 3Cob to such sequences one obtains an equivalent category. By abusing the notation, we denote this category also by 3Cob.

The symmetric monoidal structure of this category is strict monoidal with tensor product being concatenation. The symmetry arrows $\sigma_{\Xi,\Theta}\colon \Xi,\Theta\to \Theta,\Xi$, could be replaced by the following operation on arrows. Let
\[
C=(M,[\varphi_1,\ldots,\varphi_m],[\theta_1,\ldots\theta_n])
\]
be an arrow of 3Cob, whose source is  $(\Sigma_{i_1},\ldots,\Sigma_{i_m})$, and whose target is  $(\Sigma_{j_1},\ldots,\Sigma_{j_n})$. For every $k$, we have that $\varphi_k$ and $\theta_k$ are embeddings of $\Sigma_{i_k}$ and $\Sigma_{j_k}$, respectively. For $\pi$ a permutation of $\{1,\ldots,m\}$ and $\tau$ a permutation of $\{1,\ldots,n\}$, we define $C^\pi_\tau$ to be the arrow
\[
(M,[\varphi_{\pi(1)},\ldots,\varphi_{\pi(m)}],[\theta_{\tau(1)}, \ldots\theta_{\tau(n)}]),
\]
whose source is $(\Sigma_{i_{\pi(1)}},\ldots,\Sigma_{i_{\pi(m)}})$, and whose target is  $(\Sigma_{j_{\tau(1)}},\ldots,\Sigma_{j_{\tau(n)}})$.

\subsection{Diagrams for 3Cob}\label{standard diagrams}

A diagrammatic language for presenting manifolds introduced in \cite{FGOP} serves as a base for diagrammatic presentation of the arrows of 3Cob. For our purposes, this language is slightly modified. A diagram is embedded in $R^3$ and it consists of a finite set of wedges of oriented circles and a framed link (called \emph{surgery data}). The wedges of circles in a diagram are separated into a \emph{positive} and a \emph{negative} sequence. For the sake of better visualisation, we mark positive wedges in red and negative in blue (see Figure~\ref{dijagram2}). The ordering of wedges will be not indicated in the illustrations below since it is irrelevant for our examples.

\begin{figure}[!h]
\centering
\includegraphics{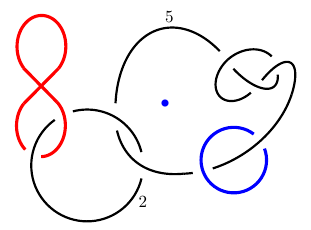}
\caption{A diagram for a cobordism}
\label{dijagram2}
\end{figure}

Let $\pi$ be the $xy$-plane. We may assume that every diagram lives in a narrow tubular neighbourhood of this plane. It is assumed that $\pi$ satisfies standard conditions listed in \cite[Paragraph preceding Figure~1.2]{PS} with respect to the diagram components (excluding the common points of circles in wedges). A diagram is projected to this plane and presented as a planar diagram in which the ``under'' and ``over'' crossings are taken with respect to $z$-coordinates.  We will not make a distinction between diagrams in $R^3$ and their planar projections. Such diagrams will be called \emph{cobordism diagrams} or just \emph{diagrams}.

We have the following geometrical assumptions concerning the wedges of $g\geq 2$ circles. Every such wedge is equipped with a ball whose center is the common point of the circles. This point is called the \emph{center} of the wedge. The parts of the circles of this wedge inside the ball are radial and parallel to $\pi$. One radius corresponding to the $i$-th circle in the wedge is outgoing (following the orientation of the circle), and the other is incoming. It is assumed that the outgoing radii corresponding to the $i$-th circle for all wedges of $g$ circles are parallel and codirected. The same holds for the incoming radii. The  pairs of radii corresponding to one circle are consecutive.

The circles belonging to one wedge are unknotted, unlinked and their projections do not cross each other. Moreover, these projections are oriented counterclockwise. In most situations, the wedges in a diagram will be unlinked as in Figure~\ref{dijagram2}. For example, the wedges of three circles are assumed to be of the form illustrated in Figure~\ref{wedges form}. However, we will not draw the balls associated with wedges in our illustrations.
\begin{figure}[!h]
\centering
    \includegraphics{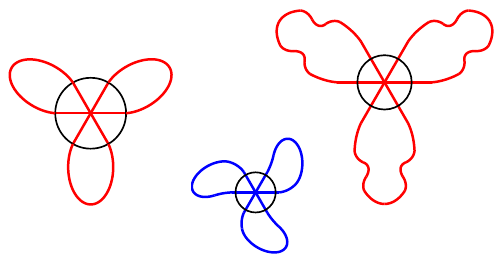}
\caption{Wedges of three circles}\label{wedges form}
\end{figure}

For every wedge $W$ of circles in a diagram one defines a handlebody $H_W$ (disjoint from the rest of the diagram and from the other such handlebodies) in the following way. If $W$ is a wedge of zero circles, then $H_W$ is a ball containing $W$ in its interior. If $W$ is a wedge of a single circle, then $H_W$ is the closure of a tubular neighbourhood of $W$. In the case when $W$ contains more than one circle and $B$ is its associated ball, then $H_W$ has $B$ as its 0-handle and its 1-handles are the closures of tubular neighbourhoods of the parts of the circles of $W$ lying outside $B$. The orientation of $H_W$ is induced by the orientation of $R^3$. The interior of $H_W$ is called the \emph{chosen neighbourhood} of $W$ (see Figure~\ref{Fig:Handlebody}). This neighbourhood corresponds to the notion of graphical neighbourhood in terminology of \cite[Definition~6]{FH19}, which is appropriate for smooth category.

\begin{figure}[!h]
\centering
\includegraphics{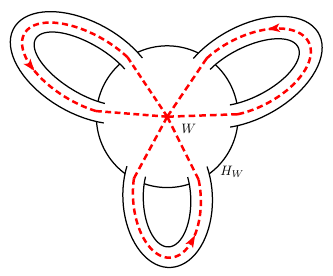}
\caption{Wedge of three circles and its chosen neighbourhood}
\label{Fig:Handlebody}
\end{figure}

For every $g\geq 0$, we fix one wedge $W_g$ of $g$ circles in $R^3$. The handlebody $H_{W_g}$, which we abbreviate by $H_g$, is obtained as above. Let the boundary of $H_g$ be the chosen representative $\Sigma_g$ of the homeomorphism class of surfaces of genus $g$, and let the orientation of $\Sigma_g$ be the opposite to its orientation induced by the orientation of $H_g$.

The surface $\Sigma_g$ is equipped with $2g$ circles, one pair, consisting of an $a$-circle and a $b$-circle, for each handle (see Figure~\ref{dijagram0}). The common point of $a$ and $b$ circles is the \emph{base point} of $\Sigma_g$. We assume that the $a$-circles belong to the 0-handle of $H_g$. By cutting $\Sigma_g$ along the $a$ and $b$-circles one obtains a polygon $\Pi_g$, with $4g$ sides, as in Figure~\ref{dijagram0}.

\begin{figure}[!h]
\centering
\includegraphics{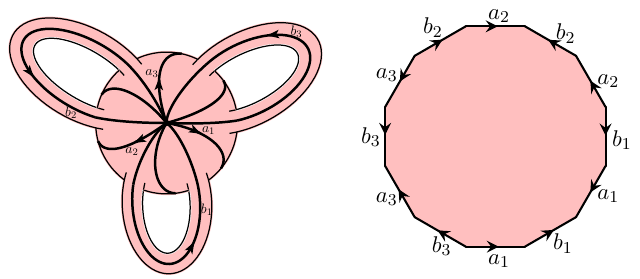}
\caption{$\Sigma_3$ and $\Pi_3$}
\label{dijagram0}
\end{figure}

The common point of $a$ and $b$ circles is the \emph{base point} of $\Sigma_g$. We assume that the $a$-circles belong to the 0-handle of $H_g$. By cutting $\Sigma_g$ along the $a$ and $b$-circles one obtains a polygon $\Pi_g$, with $4g$ sides, as in Figure~\ref{dijagram0}.

For every $g\geq 0$, we define (up to isotopy) a canonical orientation reversing homeomorphism $R_g\colon\Sigma_g\to \Sigma_g$. If $g=0$, then $R_0$ is defined as an arbitrary orientation reversing homeomorphism of $S^2$, since they are all isotopic. If $g>0$, then we use $\Pi_g$ (see Figure~\ref{dijagram0}) in order to define $R_g$. Let $R_g\colon\Sigma_g\to \Sigma_g$ be induced by a homeomorphism of $\Pi_g$ that identifies the polygonal line $a_1b_1a_1b_1a_2b_2\ldots$ with the polygonal line $b_1a_1b_1a_1b_ga_gb_ga_g\ldots$ respecting the orientation of edges (see Figure~\ref{dijagram1}). In the terminology of \cite{L88}, this is the \emph{reversion} $R=[b_1,a_1,b_g,a_g,\ldots,b_2,a_2]$.
\begin{figure}[!h]
\centering
\includegraphics{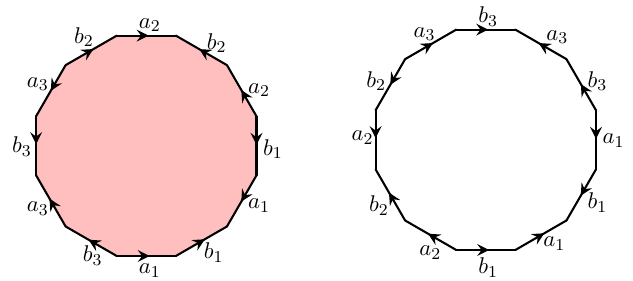}
\caption{$R_g\colon\Sigma_g\to\Sigma_g$}
\label{dijagram1}
\end{figure}

From Remark~\ref{isotopy2} it follows that a homeomorphism of $\Pi_g$ is determined, up to isotopy, by its restriction to the boundary polygonal line. Note that for $g\geq 3$ it is not the case that for every handle its $a$-circle and $b$-circle just switch the roles.

\begin{rem}\label{isotopy2}
Every homeomorphism from $D^2$ to itself which is identity on the boundary is isotopic to the identity.
\end{rem}

\subsection{The interpretation}\label{interpretation}

The advantage of a diagram in $R^3$ is that it denotes not merely a manifold and its source and target, but also an embedding of the source and an embedding of the target, i.e. a complete cobordism. Let $\mathcal{D}$ be a cobordism diagram and let $W$ be a wedge of $g$ circles in it. A homeomorphism from $(H_g,W_g)$ to $(H_{W},W)$, which preserves the orientation of handlebodies and of circles in wedges is called \emph{regular} when
\begin{enumerate}
\item it respects the handle-structure;
\item its restriction to the 0-handle is a composition of a translation and a dilation;
\item its restriction to 1-handles is such that the image of each $b$-circle has the linking number 0 with the corresponding circle of $W$.
\end{enumerate}
We call the restriction of a regular homeomorphism to the boundary of $H_g$, a \emph{regular embedding} of the surface $\Sigma_g$ into the boundary of $H_W$. This notion helps us to define the embedding of the source and the target of a cobordism presented by a diagram.

\begin{rem}\label{neighbourhood2}
By our assumption concerning the form of wedges of circles (see Figure~\ref{wedges form}), for a chosen $H_W$ and a circle embedded in $H_W-W$, we can choose inside the interior of $H_W$, a handlebody $H'_W$ containing $W$, which is disjoint from this circle. Moreover, $H'_W$ is such that there exists a regular homeomorphism from $(H_g,W_g)$ to $(H'_{W},W)$. We will use this property in Sections~\ref{handlebody} and \ref{mending}.
\end{rem}

The first step in the interpretation of $\mathcal{D}$ consists in adding the infinity point in order to place the diagram in $S^3$. Next, we remove all the chosen neighbourhoods of wedges of circles and perform the surgery according to the framed link of the diagram. As a result one obtains a connected manifold~$M$. The orientation of $M$ is induced by the orientation of $S^3$.

The source (target) of the cobordism presented by $\mathcal{D}$ is the sequence of surfaces of the form $\Sigma_g$ corresponding to the positive (negative) sequence of wedges in the diagram. The embeddings of the members of the source into the incoming boundary of $M$ are the regular embeddings, while the embeddings of the members of the target into the outgoing boundary of $M$ are the regular embeddings precomposed by corresponding reversions. (Of course, several identifications starting with the embedding of $R^3$ into $S^3$, followed by identifications of a manifold obtained by removing a solid torus from a manifold with the parts of the manifold obtained by sewing back this solid torus, are hidden in such a description of source and target embeddings.) This concludes the interpretation of cobordism diagrams as arrows of 3Cob.

\begin{rem}
If $\varphi$ and $\varphi'$ are two regular embeddings of $\Sigma_g$ into the boundary of $H_W$, then $\varphi'\circ \varphi^{-1}$ is isotopic to the identity. Hence, by Remark~\ref{isotopy1}, the choice of a regular embedding of $H_g$ does not affect the resulting cobordism.
\end{rem}

It is straightforward to see that every isotopy of $R^3$ that keeps fixed all the wedges of circles in a diagram does not affect the interpretation of this diagram. Also, every isotopy that moves just a single wedge $W$ of circles (and keeps the rest of the diagram fixed) so that in every level the 0-handle of $H_W$ is moved just by translations and dilations, does not affect the interpretation of this diagram. Note that we should always keep the counterclockwise orientation of the projections of circles in wedges.
\begin{definition}
A \emph{wedge-rigid} isotopy of $R^3$ with respect to a diagram is a composition of isotopies of the two types above.
\end{definition}

\begin{rem}\label{identity diagrams}
The diagrams from Figure~\ref{identity1} present the identity arrows on $\Sigma_g$.  We call such a configuration the \emph{identity link of wedges}. Note that according to our convention, all the linking numbers are $+1$. Moreover, the enumeration of circles in a wedge is clockwise starting with the lowest. (This enumeration is sound with the definition of the reversion from Section~\ref{standard diagrams}.)
\end{rem}

\begin{figure}[!h]
\centering
\includegraphics{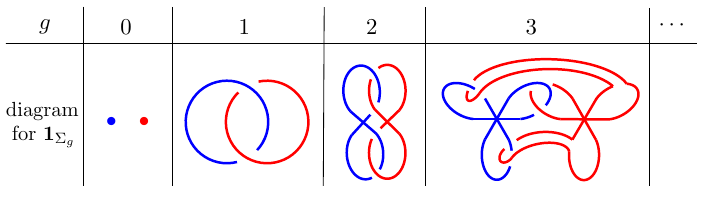}
\caption{The identity diagrams}
\label{identity1}
\end{figure}

\begin{proof}
Let $\mathcal{D}$ be an identity link of wedges with $g$ circles. As a manifold, $\mathcal{D}$ is interpreted as the complement (with respect to $S^3=R^3\cup\{\infty\}$) of the chosen neighbourhoods of these wedges. Denote this manifold by $M$. In Figure \ref{sjecenje} one can see the illustration for the case $g=2$ with images of the circles $a_i$ and $b_i$ indicated on both components of the boundary.

\begin{figure}[h!]
\centering
\includegraphics{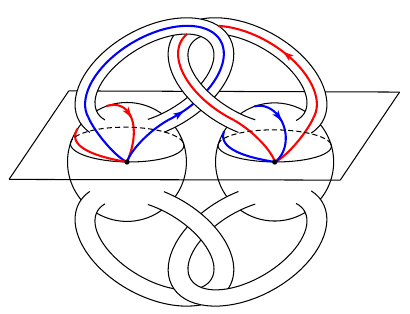}
\caption{The cutting plane}
\label{sjecenje}
\end{figure}

We have to show that there is a homeomorphism from $\Sigma_g\times I$ (see Figure~\ref{dijagram16} for the case $g=2$) to $M$ such that the component $\Sigma_g\times \{0\}$ is mapped so that the images of $a_i$ and $b_i$ are the blue and the red circle at the right-hand side of Figure \ref{sjecenje}, respectively, while the component $\Sigma_g\times \{1\}$ is mapped so that the images of $a_i$ and $b_i$ are the blue and the red circle at the left-hand side of Figure \ref{sjecenje}, respectively. We do this by cutting $\Sigma_g\times I$ by vertical cuts illustrated at the left-hand side of Figure~\ref{prizma0} (we denote the illustrated fragment by $(\Sigma_g\times I)_i$). The corresponding cut of the manifold $M$ is obtained by a plane in $R^3$ that contains the images of the base point at both boundary components (see Figure \ref{sjecenje}). After embedding into $S^3$ this plane becomes a sphere bounding a ball containing a fragment $M_i$ of $M$ illustrated at the right-hand side of Figure~\ref{prizma0}.

\begin{figure}[h!]
\centering
\includegraphics{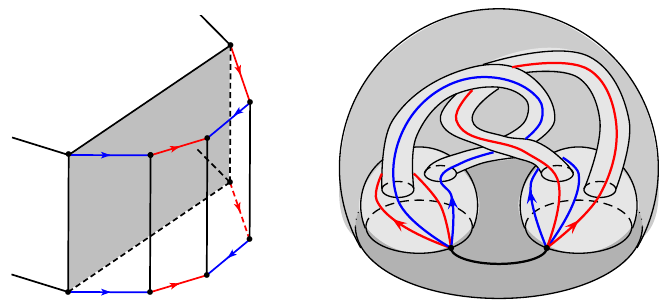}
\caption{The fragments $(\Sigma_g\times I)_i$ and $M_i$}
\label{prizma0}
\end{figure}

Consider the homeomorphism between the two images of the circle $a_i$ in the picture of $M_i$ (the blue loops based at the images of the base point), which stems from these two embeddings of $a_i$. Let $\alpha_i$ be the mapping cylinder of this homeomorphism, which is embedded in $M_i$. Analogously, let $\beta_i$ be the mapping cylinder embedded in $M_i$, which corresponds to the images of the circle $b_i$ in the picture of $M_i$ (the red loops based at the images of the base point). The embeddings of $\alpha_i$ and $\beta_i$ are such that they intersect each other just in a segment connecting the two images of the base point. This segment is illustrated in the picture of $M_i$ and it lies in the cutting plane (sphere). In the picture of $(\Sigma_g\times I)_i$, it corresponds to the vertical line segments. By cutting $M_i$ along $\alpha_i$ and $\beta_i$, one obtains a faceted ball with the same cell structure as the cell structure of $(\Sigma_g\times I)_i$. Hence, we have a homeomorphism between $(\Sigma_g\times I)_i$ and $M_i$ with the desired properties.

In the case $g>2$ one piece (the central one) of $\Sigma_g\times I$ remains uncovered by these homeomorphisms, but it contains neither $a$ nor $b$ curves. It is evidently homeomorphic to the remaining part of $M$ by a homeomorphism whose restriction to the cutting vertical rectangles coincides with the corresponding restrictions of the homeomorphisms from above. Hence, one can paste all these homeomorphisms into one with the desired properties.
\end{proof}

\begin{rem}\label{dehn1}
Let $M$ be a manifold with $\Sigma$ as a component of its boundary. Let $\theta\colon\Sigma\to\Sigma$ be a Dehn twist. A procedure introduced in \cite{L62} (explained in some details in \cite[Section~2]{FGOP}) shows how to ``immerse'' $\theta$ into $M$ in a form of a surgery along a knot. This surgery results in a manifold $M'$, with $\Sigma$ as a boundary component, for which there is a homeomorphism $h\colon M\to M'$ such that, for $\iota$ being the inclusion, the following diagram commutes:
\begin{center}
\begin{tikzcd}
\Sigma\arrow[hook]{r}{\iota}\arrow[hook]{dr}[swap]{\theta}
&M \arrow{d}{h}\\
&M'.
\end{tikzcd}
\end{center}
\end{rem}

We say that a cobordism (an arrow of 3Cob) is \emph{connected} when its underlying manifold is connected.

\begin{prop}\label{empty target}
Every connected cobordism with the empty target is presentable by a diagram with unlinked wedges.
\end{prop}

\begin{proof}
Let $M$ be a connected manifold equipped with an orientation preserving homeomorphism $\varphi\colon \Xi\to \partial M$, where $\Xi$ is an object of 3Cob. By \cite[Theorem~3.1.10]{SSS}, there exist a compression body $C$ and a handlebody $H$ such that $M$ is a result of their gluing along a homeomorphism $\theta\colon\partial H\to\partial_+ C$. The compression body $C$ could be embedded in $R^3$ so that each component of $\partial_- C$ represents the boundary of $H_W$ for some wedge $W$ and moreover, all these components are unknotted, unlinked and their projections to the $xy$-plane do not overlap. By \cite[Proposition 2.7]{FGOP} there is a diagram $\mathcal{D}$ such that $M$ could be identified with the manifold obtained by the interpretation of $\mathcal{D}$. With this identification in mind, $\mathcal{D}$ presents the cobordism $(M,\kappa)$, where $\kappa\colon \Xi\to \partial M$ consists of regular embeddings. The condition on the components of $\partial_- C$ guarantees that the wedges in $\mathcal{D}$ are unlinked.

For the homeomorphisms $\varphi,\kappa\colon \Xi\to \partial M$ let $\delta\colon \partial M\to \partial M$ be the homeomorphism such that $\kappa=\delta\circ\varphi$. After decomposing $\delta$ in Dehn twists, by iterating the application of Remark~\ref{dehn1}, one can replace them with new surgery data added to $\mathcal{D}$ in order to obtain a diagram presenting the cobordism $(M,\varphi)$.
\end{proof}

\begin{prop}
Every connected cobordism is presentable by a diagram with unlinked wedges.
\end{prop}

\begin{proof}
Let
\begin{tikzcd}
\Xi_0\arrow[hook]{r}{\varphi_0}
&M \arrow[hookleftarrow]{r}{\varphi_1} & \Xi_1
\end{tikzcd}
be an arbitrary connected cobordism. Define an orientation reversing homeomorphism $r\colon \Xi_1\to\Xi_1$ in terms of reversions $R_g\colon\Sigma_g\to\Sigma_g$ for every $\Sigma_g$ in $\Xi_1$. Consider the following arrow of 3Cob whose target is empty
\begin{center}
\begin{tikzcd}
\Xi_0\arrow[hook]{r}{\varphi_0}
&M \arrow[hookleftarrow]{r}{\varphi_1} & \Xi_1 \arrow[leftarrow]{r}{r} & \Xi_1.
\end{tikzcd}
\end{center}
By Proposition~\ref{empty target}, there exists a diagram $\mathcal{D}$ presenting this cobordism. Then, since $r\circ r$ is isotopic to the identity, one may conclude that the diagram $\mathcal{D}'$, obtained from $\mathcal{D}$ by listing the wedges of circles corresponding to $\Xi_1$ as negative, presents the initial cobordism.
\end{proof}

Every arrow of 3Cob is equal to an arrow of the form
$(C_1\otimes\ldots \otimes C_k)^\pi_\tau$, where $C_1,
\ldots,C_k$ are connected cobordisms and $\pi$, $\tau$ are permutations acting on the domain and the codomain. Hence, every arrow of 3Cob is presentable by a finite sequence of diagrams and two permutations.

\subsection{Diagrammatic calculi}\label{calculus}
It is obvious that two different cobordism diagrams may present the same arrow of 3Cob. In this section we determine the necessary and sufficient conditions under which this happens. For this we use the following notion introduced in \cite{FGOP}. Let $\Xi$ be a closed surface, which is the common boundary of manifolds $M$ and $M'$. We say that $M$ and $M'$ are $\partial$-\emph{equivalent} when there exists an orientation preserving homeomorphism $w\colon M\to M'$ such that for $\iota_M$ and $\iota_{M'}$ being the inclusions, the following diagram commutes:
\begin{center}
\begin{tikzcd}
\Xi\arrow[hook]{r}{\iota_M}\arrow[hook]{dr}[swap]{\iota_{M'}}
&M \arrow{d}{w}\\
&M'.
\end{tikzcd}
\end{center}
This means that $w$ keeps the points of $\Xi$ fixed. The main results of \cite{FGOP} showed that the diagrammatic calculi consisting of some moves, which are introduced in that paper, are complete in the sense that two such diagrams present $\partial$-equivalent manifolds if and only if there is a finite sequence of prescribed moves turning one diagram into the other (see \cite[Theorems~3.1-2 and Proposition~3.3]{FGOP}).

\begin{rem}\label{partial-eq}
Two diagrams with identical sequences of positive and negative wedges of circles present the same arrow of 3Cob if and only if the manifolds presented by these diagrams are $\partial$-equivalent.
\end{rem}

\begin{proof}
Let $M$ and $M'$ be the manifolds presented by two such diagrams. Let $W$ be a wedge of $g$ circles shared by these diagrams, and let $\Xi$ be the boundary of $H_W$. So, $\Xi$ is a common boundary component of $M$ and $M'$. Denote by $\varphi$ a regular embedding of $\Sigma_g$ into $\Xi$, and by $\iota_M$ and $\iota_{M'}$ the inclusions of $\Xi$ in $M$ and $M'$, respectively. This defines the embeddings $\varphi_M=\iota_M\circ\varphi$ and $\varphi_{M'}=\iota_{M'}\circ\varphi$ of $\Sigma_g$ into $M$ and $M'$, respectively. Then, for a homeomorphism $w\colon M\to M'$, the following holds
\[
\varphi_{M'}=w\circ\varphi_M\quad \Leftrightarrow \quad \iota_{M'}=w\circ\iota_M.
\]
By repeating this argument for all wedges, one obtains the above equivalence.
\end{proof}

As a corollary of Remark~\ref{partial-eq} and the results obtained in \cite{FGOP}, we have the following.

\begin{prop}\label{completeness}
Two diagrams with unlinked wedges, and with identical sequences of positive and negative wedges of circles present the same arrow of 3Cob if and only if one of the following three conditions hold: there is a finite sequence of moves
\begin{enumerate}
\item (-1), (0) and (1) from \emph{Figure~\ref{d2}},
\item M1-M5, W1-W4, \cite[Figures~18-19]{FGOP},
\item (-1), (2) and (1), \cite[Figure~27]{FGOP},
\end{enumerate}
transforming one into the other.
\end{prop}

\begin{figure}[!h]
\centering
\includegraphics{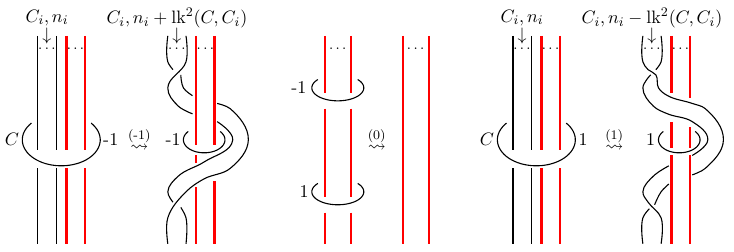}
\caption{The moves (-1), (0) and (1)} \label{d2}
\end{figure}

We will not go into details about the moves listed above since this is not the main subject of this paper and everything is thoroughly explained in \cite{FGOP}. Only the moves (-1), (0) and (1) are illustrated in Figure~\ref{d2}. (The red threads in this figure represent both red and blue threads of our diagrams.) The first two conditions in Proposition~\ref{completeness} treat the integer surgery calculus, save that the second presents a finite list of local moves sufficient for the completeness result. The third condition is devoted to the rational calculus.

Note that Proposition~\ref{completeness} compares only diagrams with identical sequences of positive and negative wedges of circles. The reason is that all the moves listed above keep the wedges of circles fixed. The formulation of Proposition~\ref{completeness} does not limit its application. If two diagrams (presenting arrows from the same hom-set in 3Cob) have no identical sequences of positive and negative wedges, one can use a wedge-rigid isotopy to make the corresponding wedges coincide. The initial diagrams present the same arrow of 3Cob if and only if the new diagrams are such. It remains to apply Proposition~\ref{completeness} to the new diagrams.

The move Twist illustrated in Figure~\ref{twist}, which does not keep the wedges of circles fixed, helps us to link a blue and a red wedge in a form that we will use for composing diagrams. It is straightforward to check that this move does not change the interpretation.
\begin{figure}[h!]
\centering
\includegraphics{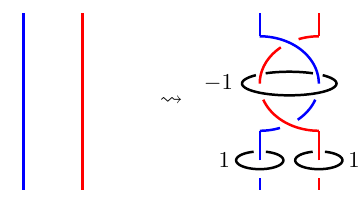}
\caption{The move Twist}\label{twist}
\end{figure}

This finishes our characterisation of the arrows of 3Cob in terms of cobordism diagrams. However, the main goal of this paper is to provide a procedure for composing diagrams, i.e., for two diagrams of composable arrows of 3Cob, to find a diagram of their composition. This goal requires several auxiliary steps.

\subsection{Diagrams within handlebodies and thick surfaces}\label{handlebody}

Apart from standard cobordism diagrams introduced in Section~\ref{standard diagrams}, we need diagrams consisting again of surgery data, but now placed within a handlebody, whose boundary also counts as a component of the outgoing boundary. More precisely, assume that for a wedge $W$ of $g$ circles, not belonging to a diagram $\mathcal{D}$, we have that this diagram is contained in the interior of $H_W$. Then the pair $(H_W,\mathcal{D})$ makes a diagram within a handlebody. Since the boundary of $H_W$ counts as a part of the outgoing boundary, we must indicate by a label its place in the target. Such a diagram is illustrated at the left-hand side of Figure~\ref{handlebody1} (see also the shaded region of the diagram at the right-hand side of Figure~\ref{CD}).
\begin{figure}[!h]
\centering
\includegraphics{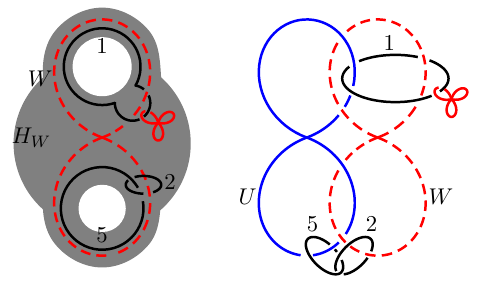}
\caption{A diagram within a handlebody} \label{handlebody1}
\end{figure}

A diagram within a handlebody is interpreted so that we start with $H_W$, remove the chosen neighbourhoods of the wedges of circles and perform the surgery according to the framed link in $\mathcal{D}$. This results in a connected manifold. The orientation of this manifold is induced by the orientation of~$H_W$. The source and the target are defined as in Section~\ref{interpretation}, save that $\Sigma_g$ is added to the target list at the place indicated by the label assigned to $H_W$ in the diagram. The embeddings of the members of the source and the target into the boundary of the manifold are defined as in Section~\ref{interpretation}, save that the embedding of $\Sigma_g$, which is added to the target list, is its regular embedding in the boundary of $H_W$. (Note that this embedding is orientation reversing.) In the case of our example above, we start with a handlebody with two handles---the chosen neighbourhood of the dashed wedge of two circles. Then we remove the chosen neighbourhood of the wedge of three circles and perform surgery along the framed link with three components. A regular embedding of the source $\Sigma_3$ into the boundary of the chosen neighbourhood of the wedge of three circles is orientation preserving and a regular embedding of the target $\Sigma_2$ into the boundary of the ambient handlebody is orientation reversing.

Our goal is to transform a given standard cobordism diagram into a diagram within a handlebody presenting the same cobordism. For example, the standard diagram at the right-hand side of Figure~\ref{handlebody1} (the dashed red wedge should be neglected---its role will become clear in a moment) is transformed into the diagram within a handlebody illustrated at the left-hand side of the same figure.
We call this procedure \emph{inside-out} and it relies on the following remark.

\begin{rem}\label{izvrtanje1}
Let $\mathcal{D}$ be a diagram, which contains an identity link $F$ of wedges (see Remark~\ref{identity diagrams}). Let $U$ be the blue, and $W$ be the red wedge in $F$. Assume that the other wedges in $\mathcal{D}$ are unlinked. Then $\mathcal{D}-U$ could be placed in the interior of $H_W$ by using only wedge-rigid isotopy.
\end{rem}

\begin{proof}
One can start by pulling $F$ (together with the threads of the framed link hanging on the circles in $F$) out of $\mathcal{D}$. See the left-hand side of Figure~\ref{Fig:Pulling}.
\begin{figure}[!h]
\centering
\includegraphics{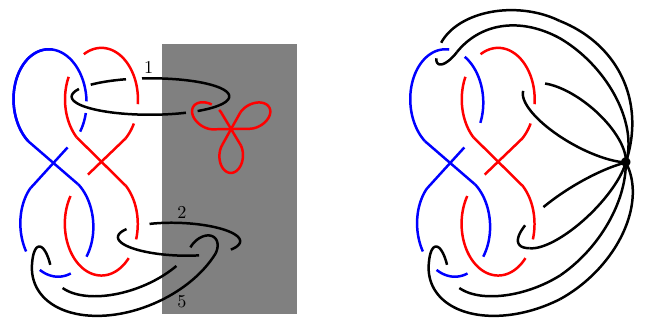}
\caption{}
\label{Fig:Pulling}
\end{figure}

Note that the fundamental group of the complement of $F$ with respect to $R^3$ is generated by the four generators illustrated at the right-hand side of Figure~\ref{Fig:Pulling}. (In the case of wedges of $g$ circles in $F$, there are $2g$ generators of this fundamental group.) By relying essentially on this fact, all the threads coming out of the box (see the left-hand side of Figure~\ref{Fig:Pulling}) and passing through the circles in $F$, could be decomposed into the pieces corresponding to these four generators (see the left-hand side of Figure~\ref{Fig:Decomposition}).
\begin{figure}[!h]
\centering
\includegraphics{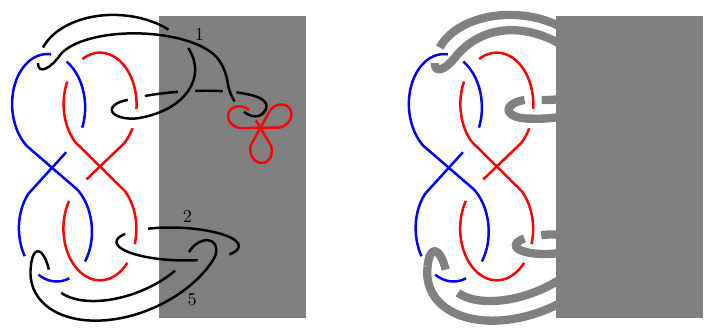}
\caption{}
\label{Fig:Decomposition}
\end{figure}

By doing this, $\mathcal{D}$ could be placed so that $\mathcal{D}-F$ belongs to the shaded region in the right-hand side of Figure~\ref{Fig:Decomposition}. From this, it is clear that $\mathcal{D}-U$ could be placed in the interior of $H_W$. All the moves from above are covered by wedge-rigid isotopy.
\end{proof}

Let $\mathcal{D}''$ be the diagram within $H_W$ obtained by Remark~\ref{izvrtanje1}. By removing the wedge $W$ from $\mathcal{D}''$ we get a diagram $\mathcal{D}'$. Let the label assigned to $H_W$ in $\mathcal{D}'$ be the same as the label assigned to $U$ in $\mathcal{D}$.

\begin{prop}\label{izvrtanje2}
The diagram $\mathcal{D}'$ within the handlebody $H_W$ presents the same cobordism as $\mathcal{D}-W$.
\end{prop}
\begin{proof}
By Remark~\ref{identity diagrams} the diagram consisting of only $U$ and $W$, presents the identity cobordism. Its gluing with the cobordism presented by $\mathcal{D}'$ along $\partial H_W$ results in a cobordism presented by $\mathcal{D}-W$. Hence, $\mathcal{D}-W$ and $\mathcal{D}'$ present the same cobordism.
\end{proof}

The inside-out procedure goes as follows. Consider a diagram with unlinked wedges containing a wedge $U$ of blue circles. Let $\mathcal{D}$ be a result of adding a wedge $W$ of red circles to this diagram, so that $U$ and $W$ form an identity link of wedges, which is unlinked with the other wedges. Then apply the procedure from Remark~\ref{izvrtanje1}. By Proposition~\ref{izvrtanje2}, the obtained diagram $\mathcal{D}'$ within $H_W$, presents the same cobordism as the initial diagram.

\vspace{1ex}

Consider now the diagram $\mathcal{D}''$ within $H_W$ obtained by Remark~\ref{izvrtanje1}. We choose $H'_W$, according to Remark~\ref{neighbourhood2}, so that there exists a regular homeomorphism from $(H_g,W_g)$ to $(H'_{W},W)$ and that $H'_W$, besides $W$, does not intersect the rest of $\mathcal{D}''$. Let us interpret $\mathcal{D}''$ as a diagram within a handlebody, save that we also remove the interior of $H'_W$.  The embedding of $\Sigma_g$ into the boundary of $H'_W$ is just the restriction of the regular homeomorphism from above. Note that, this time, it is orientation preserving. By arguing as in the proof of Proposition~\ref{izvrtanje2} we have the following.

\begin{prop}\label{izvrtanje3}
The diagram $\mathcal{D}''$  presents the same cobordism as $\mathcal{D}$.
\end{prop}

One may envisage the diagram $\mathcal{D}''$ as a collection of wedges and a framed link within the interior of $H_W-H'_W$.  We call such a diagram, a \emph{diagram within a thick surface}. If it is empty, i.e.,\ it contains no surgery data, then it presents the identity arrow $\mj_{\Sigma_g}$.

\section{A surgery for $\Sigma_g\times S^1$}\label{sigmags1}

For our purposes it is important to find a surgery presentation for closed manifolds of the form $\Sigma_g\times S^1$. It is well known that when $g=0$, i.e., when we deal with $S^2\times S^1$, this manifold could be presented by an unknot with framing 0. A bit less familiar are the cases of 3-dimensional torus ($g=1$), which could be presented by the Borromean rings with 0-framing of each component (see \cite[Exercises 5.3.3(d) and 5.4.3(c)]{GS99}), and the general case for arbitrary $g$, which is mentioned in \cite[Proof of Theorem~8]{L93}. We will work out here the general case.

Let us illustrate the case $g=2$, which suffices to clarify the matters. We start with the octagonal prism $P$ illustrated in Figure~\ref{dijagram4}. Let the face-pairing $\epsilon$ be such that the bases of this prism are identified as well as the pairs of lateral facets having the same labels of edges (the vertical edges are supposed to have the same label). The quotient cell complex $P/\epsilon$ corresponds to the manifold $\Sigma_2\times S^1$.
\begin{figure}[h!]
\centering
\includegraphics[width=5.4cm]{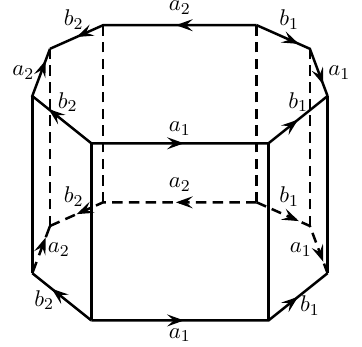}
\caption{Face pairing for $\Sigma_2\times S^1$} \label{dijagram4}
\end{figure}

By using the terminology established in \cite[Section~3]{CFP03} there is a chimney assembly for $P$ with 8 quadrilateral $f$-chimneys inside the lateral facets and two octagonal $f$-chimneys inside the bases. By identifying appropriate $f$-chimneys, one obtains a handlebody with 5 handles, which is a part of a genus five Heegaard splitting for $\Sigma_2\times S^1$.

\begin{definition}
Let $\Xi$ be a connected, orientable surface of genus $g$. A \emph{system of attaching circles} for $\Xi$ is
a set $\{\gamma_1,\ldots,\gamma_g\}$ of simple closed curves on this surface such that:
\begin{enumerate}
\item the curves $\gamma_i$ are mutually disjoint,
\item $\Xi-\gamma_1-\ldots-\gamma_g$ is connected.
\end{enumerate}
Let $\alpha$ and $\beta$ be two systems of attaching circles for $\Xi$. The triple $(\Xi,\alpha,\beta)$ is a \emph{Heegaard diagram}.
\end{definition}

Since the quotient complex $P/\epsilon$, that
is a manifold, has only one vertex, we can apply \cite[Theorem 4.2.1]{CFP03} in order to obtain a Heegaard diagram for this manifold (for the most symmetric form of this diagram see Figure~\ref{dijagram5}).
\begin{figure}[!h]
\centering
\includegraphics[width=5.5cm]{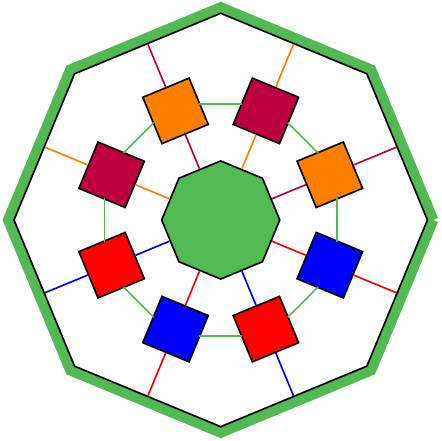}
\caption{A planar Heegaard diagram for $\Sigma_2\times S^1$} \label{dijagram5}
\end{figure}

This diagram is placed on $S^2$ and the Heegaard surface of genus five is obtained by cutting out the coloured regions of the diagram and by identifying their boundaries in an orientation reversing manner induced by the above face-pairing $\epsilon$. One system of attaching circles is made of the identified boundaries of coloured regions, and the other is obtained by concatenating the coloured line segments. A picture of this genus five surface (embedded in $R^3$) together with the Heegaard diagram is complicated but could provide some insight to a careful reader. For example, a Heegaard diagram for $T^3=\Sigma_1\times S^1$ on a genus 3 surface embedded in $R^3$ is illustrated in Figure~\ref{dijagram6}.
\begin{figure}[h!]
\centering
\includegraphics{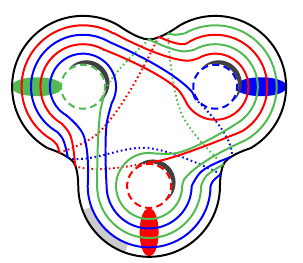}
\caption{Heegaard's diagram for $T^3$}
\label{dijagram6}
\end{figure}

Assume that the diagram from Figure~\ref{dijagram5} is placed on the boundary of a ball. The identification of coloured regions corresponds to adding 1-handles to this ball. In this way one obtains a handlebody, which we denote by $H_2$. We assume that $H_2$ is standardly embedded in $R^3$ and its complement with respect to $S^3=R^3\cup\{\infty\}$ is the handlebody $H_1$.

\begin{definition}\label{meridian}
Let $H$ be a handlebody with $g$ handles. A \emph{meridional disk} is a properly embedded disk in $H$ (its boundary belongs to the boundary of $H$) such that by removing its neighbourhood from $H$ one obtains a handlebody with $g-1$ handles. A \emph{complete system of meridional disks} for $H$ consists of $g$ mutually disjoint disks such that by removing their neighbourhoods from $H$ one obtains a ball.
\end{definition}

Let us concentrate on a half of the diagram from Figure~\ref{dijagram5}. It gives the pattern illustrated in Figure~\ref{dijagram7}, which occurs repeatedly $g$-times in the case of the Heegaard diagram for $\Sigma_g\times S^1$.
\begin{figure}[h!]
\centering
\includegraphics[width=8.8cm]{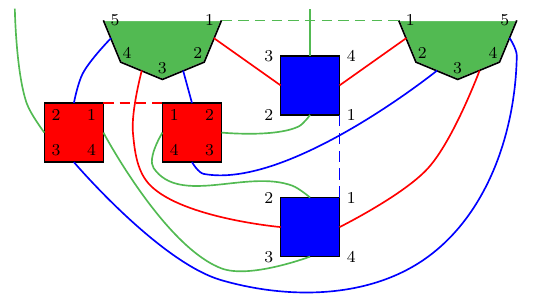}
\caption{The pattern in Heegaard's diagram for $\Sigma_g\times S^1$} \label{dijagram7}
\end{figure}

Let us denote by $\alpha$ the system of attaching circles made of boundaries of coloured regions, and by $\beta$ the system of coloured attaching circles. Let $\gamma$ be the system of attaching circles made of dashed line segments. It is evident that the circles in $\alpha$ bound a complete system of meridional disks in $H_2$ and that the circles in $\gamma$ bound a complete system of meridional disks in $H_1$. Therefore, $(\Xi,\alpha,\gamma)$ is a Heegaard diagram for $S^3$, while $(\Xi,\alpha,\beta)$ is a Heegaard diagram for $\Sigma_g\times S^1$. Both diagrams are of genus $2g+1$. For a surgery presentation of $\Sigma_g\times S^1$, it is important to find Dehn's twists, which map (up to isotopy) the circles from $\gamma$ to the circles from $\beta$.

In order to see how the circles from $\beta$ are linked in $S^3$, one has to make ``bridges'' in Figure~\ref{dijagram7} by using the attached handles. For the pattern illustrated in that figure one obtains a fragment of a link, which is isotopic in $S^3$ to the one illustrated in Figure~\ref{dijagram8}.
\begin{figure}[h!]
\centering
\includegraphics{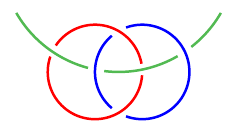}
\caption{The pattern in the link} \label{dijagram8}
\end{figure}

Let $\mathcal{L}$ be the link of circles from $\beta$ in Heegaard's diagram for $\Sigma_g\times S^1$. We conclude that it consists of one distinguished circle, which we call \emph{Brunnian} (the green one in Figure~\ref{dijagram8}), and $2g$ circles coming in pairs, which we call \emph{coupled}. Each pair of coupled circles is linked to the Brunnian one in the form of Borromean rings. The following remark will be used in Section~\ref{mending}.

\begin{rem}\label{positioning}
If we consider the link $\mathcal{L}$ up to isotopy in $S^3-{\rm int}(H)$, where $H$ is the handlebody obtained from $H_2$ by removing the 1-handle with green attaching regions, then each of its patterns is of the form illustrated in Figure~\ref{dijagram9}. Hence the Brunnian circle is unlinked with $H$, while the coupled circles go through the corresponding handles of $H$.
\end{rem}
\begin{figure}[h!]
\centering
\includegraphics[width=6.1cm]{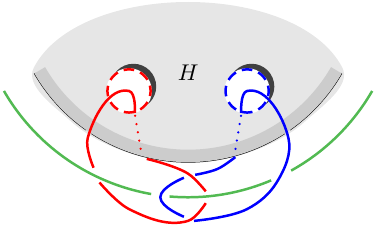}
\caption{}
\label{dijagram9}
\end{figure}

Let $\beta=\{\beta_1,\ldots,\beta_{2g+1}\}$ and $\gamma=\{\gamma_1,\ldots,\gamma_{2g+1}\}$ in the above diagrams be such that for every $i\in\{1,\ldots,2g+1\}$, the circles $\beta_i$ and $\gamma_i$ are of the same colour. Let $\theta_l$ denote Dehn's twist of the Heegaard surface $\Xi$ along the circle $l$. By \cite[Lemma~1]{L62} it follows that the composition
\begin{equation}\label{theta}
\theta=\theta_{\gamma_{2g+1}}\circ \theta_{\beta_{2g+1}}\circ\ldots\circ \theta_{\gamma_1}\circ \theta_{\beta_1},
\end{equation}
which is a self homeomorphism of $\Xi$, maps for every $i\in\{1,\ldots,2g+1\}$ the circle $\gamma_i$ onto $\beta_i$. Therefore, the manifold
$\Sigma_g\times S^1$ is homeomorphic to $(H_1\sqcup H_2)/\sim$, where $\sim$ is such that
\[
   \forall x\in\Xi \quad (x,1)\sim (\theta(x),2).
\]

By the procedure introduced in \cite{L62}, one can replace $(H_1\sqcup H_2)/\sim$ by a surgery in $S^3$ with respect to the link obtained by immersing the curves $\beta_1,\gamma_1,\ldots,\beta_{2g+1},\gamma_{2g+1}$ into $H_1$ level by level so that $\beta_1$ is the deepest. The framing of this link is calculated as follows. Let $l$ be a member of $\beta\cup \gamma$. Choose an orientation of $l$ and introduce a curve $l'$ on $\Xi$, codirected with $l$ such that $l$ and $l'$ bound the annulus where the Dehn twist $\theta_l$ performs. The framing of the link component obtained by immersing $l$ into $H_1$ is by 1 greater than the linking number of $l$ and $l'$. In the above case the framing of every component is 1, and one may imagine that the circles from $\beta$ remain at the same places, while the circles from $\gamma$ are shallowly immersed in $H_2$.

In this way, we obtain a framed link consisting of the link $\mathcal{L}$ introduced above, whose each component is linked with an unknot and all the components have the framing 1. By Kirby's calculus we have
\begin{figure}[!h]
\centering
\includegraphics[width=3.5cm]{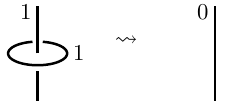}
\end{figure}

\noindent and the link components corresponding to members of $\gamma$ could be cancelled out, leaving the link $\mathcal{L}$ with zero framing of every component. Hence, a surgery presentation of $\Sigma_g\times S^1$ is the one illustrated in Figure~\ref{dijagram10}.
\begin{figure}[h!]
\centering
\includegraphics[width=7.8cm]{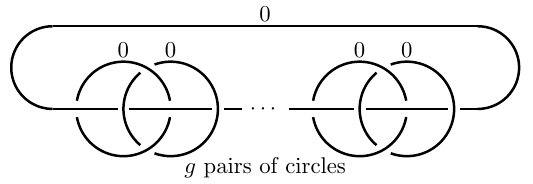}
\caption{$\Sigma_g\times S^1$ surgery} \label{dijagram10}
\end{figure}

\section{How to compose diagrams?}\label{composition}

First, we have to reduce the composition of arrows of 3Cob to an operation on connected arrows of this category. As we noted, 3Cob is symmetric, strictly monoidal with concatenation as the tensor product. Also, every arrow of 3Cob is equal to an arrow of the form $(C_1\otimes\ldots \otimes C_k)^\pi_\tau$, where $C_1,
\ldots,C_k$ are connected arrows and $\pi$, $\tau$ are permutations acting on the domain and the codomain. By bifunctoriality of the tensor, this arrow could be ``developed'' in the form
\begin{equation}\label{developed}
((C_1\otimes\mj)\circ(\mj\otimes C_2\otimes\mj)\circ\ldots\circ (\mj\otimes C_k))^\pi_\tau,
\end{equation}

If one has to compose two arrows of the form~\ref{developed}, then by naturality of symmetry and bifunctoriality of tensor, it suffices to learn how to compose two arrows of the form $C\otimes\mj$ and $\mj\otimes D$ such that
\begin{itemize}
\item[$(\dagger)$] $C$, $D$ and $(\mj\otimes D)\circ(C\otimes\mj)$ are connected.
\end{itemize}
Figures~\ref{dijagram11} and \ref{dijagram12} should convince the reader that this holds. (The cobordisms $C'$ and $D'$ have the same underlying manifolds as $C$ and $D$ respectively, just the embeddings of the targets and the sources are permuted.)
\begin{figure}[h!]
\centering
\includegraphics{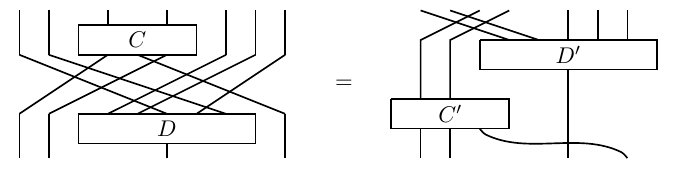}
\caption{$C$ and $D$ not confronted} \label{dijagram11}
\end{figure}
\vspace{-0.3cm}
\begin{figure}[h!]
\centering
\includegraphics{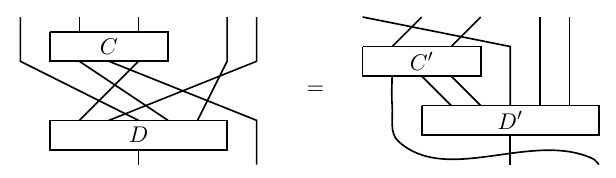}
\caption{$C$ and $D$ confronted} \label{dijagram12}
\end{figure}

The condition $(\dagger)$ reduces the problem of composition of two cobordisms presented by diagrams to the case of two diagrams $\mathcal{D}_C$ and $\mathcal{D}_D$, and a collection of blue wedges in $\mathcal{D}_C$, and the same collection of red wedges in $\mathcal{D}_D$. One has to find a diagram that presents the cobordism obtained by gluing $C$ and $D$ along the boundary components corresponding to these collections of wedges. If our collections consist just of a single wedge, i.e.\ when the gluing surface is connected, then this corresponds to the binary partial operation of \emph{sewing} (in terminology of \cite{S04}) on diagrams. We describe this situation first.

Otherwise, when the gluing surface is not connected, then besides sewing, one needs the unary partial operation of \emph{mending} (in terminology of \cite{S04}) on diagrams. Let us assume that the chosen collection of wedges contains $k>1$ wedges. In that case, a sewing operation is performed first, in order to obtain a diagram $\mathcal{D}$ presenting the cobordism obtained from $C$ and $D$ by gluing along a boundary component corresponding to one member of the chosen collection. It remains to apply $k-1$ operations of mending to the diagram $\mathcal{D}$ in order to obtain the diagram for $(\mj\otimes D)\circ(C\otimes\mj)$.

\subsection{Sewing two diagrams} Suppose we have diagrams $\mathcal{D}_C$ and $\mathcal{D}_D$ with unlinked wedges. Let $C$ and $D$ be cobordisms presented by these diagrams. Let $U$ be a blue wedge in $\mathcal{D}_C$, and let $V$ be a red wedge with the same number of circles in $\mathcal{D}_D$. By gluing $C$ and $D$ along the boundary components corresponding to $U$ and $V$, we obtain a cobordism, and our goal is to find a diagram that presents it. As an illustration, one can take the left-hand side diagram of Figure~\ref{CD} as $\mathcal{D}_D$ and the right-hand side diagram of the same figure as $\mathcal{D}_C$.

First, we apply the inside-out procedure to $\mathcal{D}_C$ with respect to $U$. In Figure~\ref{CD} this gives the diagram within the handlebody indicated by shaded region in $\mathcal{D}_C$. Next, by relying on regular homeomorphisms from $H_2$, to $H_V$ and to the shaded handlebody, we move the diagram within this handlebody to $H_V$, and remove $V$. This results in the left-hand side diagram of Figure~\ref{E}. Finally, in cases with more wedges than in our example, a reordering of wedges left after this sewing must be performed. Namely, the sequence of red (blue) wedges from $D$ should be appended to the sequence of red (blue) wedges in $C$.

\subsection{Mending a diagram}\label{mending} Suppose we have a diagram with unlinked wedges, with a chosen pair (one blue and one red) of wedges with $g$ circles in it. By mending the cobordism presented by this diagram along the incoming and the outgoing boundary component corresponding to the chosen wedges, we obtain another cobordism $C$. Our goal is to find a diagram that presents $C$. As an illustration, one can take the diagram at the left-hand side of Figure~\ref{D} as the initial (this figure is the same as Figure~\ref{E} from Introduction, save that some labels are added).
\begin{figure}[!h]
\centering
\includegraphics[width=12.1cm]{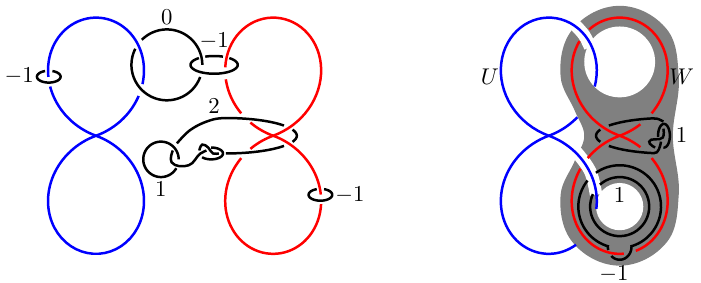}
\caption{The initial diagram and the diagram $\mathcal{D}$}
\label{D}
\end{figure}

First, by applying the move Twist illustrated in Figure~\ref{twist}, the diagram is transformed so that the chosen wedges make an identity link of wedges (see Remark~\ref{identity diagrams}). This could be done in a way that all but these two wedges remain unlinked. Let us denote this diagram by $\mathcal{D}$, and let $U$ and $W$ be the blue and the red wedge making the identity link in $\mathcal{D}$. The initial diagram and the diagram $\mathcal{D}$ present the same cobordism. As an illustration of $\mathcal{D}$ one can take the diagram at the right-hand side of Figure~\ref{D} (just neglect the shading for a moment). Note that moves from Figure~\ref{d2} are used to simplify this diagram.

Next, we apply the procedure from Remark~\ref{izvrtanje1} with respect to this identity link of wedges in order to obtain a diagram $\mathcal{D}''$ within a thick surface (see the end of Section~\ref{handlebody}). The diagram $\mathcal{D}''$ in our example is placed in the shaded handlebody $H_W$ at the right-hand side of Figure~\ref{D}. The thick surface is obtained by removing the interior of $H'_W$ from this handlebody.

We envisage $\mathcal{D}''$ as a surgery data within the interior of $H_W-H'_W$. By Proposition~\ref{izvrtanje3}, the initial diagram and $\mathcal{D}''$ present the same cobordism. The images of the base point of $\Sigma_g$ on the boundaries $\partial H_W$ and $\partial H'_W$ could be connected by a segment. Let us call a tubular neighbourhood of this segment the \emph{channel}. By a wedge-rigid isotopy, the diagram $\mathcal{D}''$ could be dislocated from the channel. The manifold obtained by removing the interior of $H'_W$ and the channel from $H_W$ is a handlebody with $2g$ handles (see Figure~\ref{dijagram17}).
\begin{figure}[h!]
\centering
\includegraphics{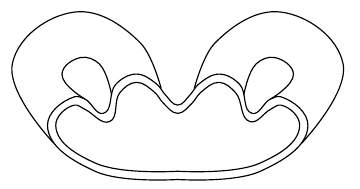}
\caption{} \label{dijagram17}
\end{figure}
\begin{thm}
The cobordism $C$ is presented by a diagram obtained from $\mathcal{D}''$ by neglecting the boundaries $\partial H_W$ and $\partial H'_W$ and by adding a framed link of the form illustrated in Figure~\ref{dijagram10}, save that the coupled circles of this link should be placed as the blue and red circles in Figure~\ref{dijagram18}.
\end{thm}

\begin{proof}
We identify the thick surface obtained by removing the interior of $H'_W$ from $H_W$, with $\Sigma_g\times I$ obtained from a prism by identifying pairs of its lateral facets. For example, the case $g=2$ is illustrated in Figure~\ref{dijagram16}, where the lateral facets having the same labels of edges are identified. Such a prism was used in Section~\ref{sigmags1} to describe manifolds $\Sigma_g\times S^1$.
\begin{figure}[h!]
\centering
\includegraphics[width=5.6cm]{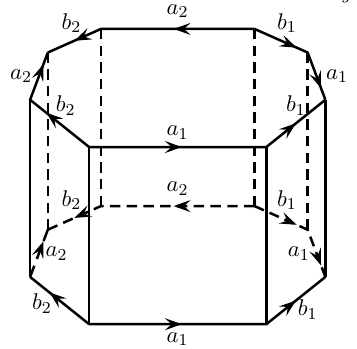}
\caption{Thick $\Sigma_2$} \label{dijagram16}
\end{figure}

Since the surgery data of $\mathcal{D}''$ lie in the interior of $\Sigma_g\times I$, it can be detached from the basis of the prism. Moreover, by wedge-rigid isotopy it can be moved further from the vertical edges, which are all identified into a segment connecting the base points of the two copies of $\Sigma_g$ (the core of the channel). Hence, we may assume that the surgery data of $\mathcal{D}''$ is placed in the interior of the handlebody $H$ introduced in Remark~\ref{positioning}. This handlebody with $2g$ handles is obtained by shaving the bases of the prism (this removes one 1-handle) and by removing the channel. An illustration of $H$ with the channel broaden a bit is given in Figure~\ref{dijagram17}.

By adding one more handle to $H$, we obtain the handlebody $H_2$ introduced in Section~\ref{sigmags1}, Paragraph preceding Definition~\ref{meridian}. The complement of $H_2$ with respect to $S^3$ is the handlebody $H_1$, and $\Xi$ is the common boundary of $H_1$ and $H_2$.

\begin{figure}[h!]
\centering
\includegraphics{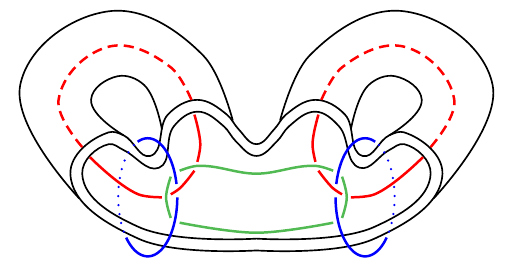}
\caption{} \label{dijagram18}
\end{figure}

The manifold underlying the cobordism $C$ is obtained in the following manner. First identify the bases of the prism in order to obtain $\Sigma_g\times S^1$, and then remove all the chosen neighbourhoods of wedges of circles and perform the surgery according to the diagram $\mathcal{D}''$. The first step of this procedure could be replaced by gluing the handlebodies $H_1$ and $H_2$ so to obtain $\Sigma_g\times S^1$, and this is achieved as in Section~\ref{sigmags1}. The gluing homeomorphism $\theta\colon\Xi\to\Xi$ is equal to the composition of Dehn's twists as in~\ref{theta}.

Note that the curves $\beta$ participating in Heegaard's diagram for $\Sigma_g\times S^1$ from Section~\ref{sigmags1} are immersed into $H_1$, forming a link as the one in Figure~\ref{dijagram10}. The relationship between this link and the handlebody $H$ is described in Remark~\ref{positioning}. This means that with respect of $H$ illustrated in Figure~\ref{dijagram17}, this link is placed as in Figure~\ref{dijagram18}. As calculated in Section~\ref{sigmags1}, every component of this link has 0 as framing. This new surgery in $S^3$ replaces the above gluing of $H_1$ and $H_2$, and the cobordism $C$ is obtained as the interpretation of the surgery data from $\mathcal{D}''$ (placed in the interior of $H$) and the framed link illustrated in Figure~\ref{dijagram10} placed as in Figure~\ref{dijagram18}.
In the example given in Figure~\ref{D} this results in the diagram illustrated in Figure~\ref{F}.
\end{proof}

\bigskip
\noindent\textbf{Funding} The authors were supported by the Science Fund of the Republic of Serbia,
Grant No. 7749891, Graphical Languages - GWORDS.
\bigskip
\\
\noindent\textbf{Data Availability} This work has no associated data.
\bigskip
\\
\noindent\textbf{Declarations}
\\
\noindent\textbf{Conflict of interest} The authors have no relevant financial or non-financial interests to disclose.

\end{document}